\pdfoutput=1
\documentclass[12pt]{article}
\usepackage[utf8]{inputenc}
\usepackage[margin = 1in]{geometry}

\usepackage{graphicx}
\usepackage[all,cmtip]{xy}
\usepackage{subfigure,epsfig}
\usepackage[shortlabels]{enumitem}
\usepackage{amsmath}
\usepackage{amssymb}
\usepackage{amsthm}
\usepackage{bbm}
\usepackage[english]{babel}
\usepackage{fancyhdr}
\usepackage{stmaryrd}
\usepackage{tikz-cd}
\usepackage{mathtools}
\usepackage{titling}
\usepackage{dynkin-diagrams}
\usepackage[new]{old-arrows}
\usepackage{mathrsfs}
\usepackage[style=ext- alphabetic,sorting=nyt,maxnames=50,firstinits=true,maxalphanames=10,useprefix=true,articlein=false]{biblatex}

\usepackage[colorlinks=true, allcolors=blue]{hyperref}

\newtheorem{Theorem}{Theorem}[section]
\newtheorem{thm}{Theorem}[section]
\newtheorem{Corollary}[thm]{Corollary}

\newtheorem{Lemma}[thm]{Lemma}
\newtheorem{Conjecture}[thm]{Conjecture}

\theoremstyle{definition}

\theoremstyle{remark}

\title{Minimal Nilpotent Orbits of type $D$ and $E$}
\author{Boming Jia\thanks{\vspace{-0.1em}The author was supported by NSFC Grant No. 12225108 and the Shuimu Scholar Program in Tsinghua University.}}

\date{}

\newcommand{\Wedge}{\scalebox{0.8}{\raisebox{0.4ex}{$\bigwedge$}}}
\newcommand{\slc}{\mathfrak{sl}_3}
\newcommand{\sld}{\mathfrak{sl}_4}
\newcommand{\sln}{\mathfrak{sl}_n}

\newcommand{\so}{\mathfrak{so}}

\newcommand{\Lieg}{\mathfrak{g}}

\newcommand{\Liep}{\mathfrak{p}}

\newcommand{\suchthat}{\,|\,}

\newcommand{\Sym}{\mathrm{Sym}}

\DeclareMathOperator{\Hom}{Hom}
\DeclareMathOperator{\Spec}{Spec}

\newcommand{\C}{\mathbb{C}}

\newcommand{\Omin}{\mathcal{O}_{\textrm{min}}}

\newcommand{\Ominbar}{\overline{\mathcal{O}}_\textrm{min}}

\begin{document}

\setlength{\droptitle}{-6em}	

\maketitle

\vspace{-4em}

\begin{abstract}
We first show the closure of the minimal nilpotent adjoint orbit $\Omin^{D_n}$ in $\so_{2n}$ is isomorphic to the affinization of $T^*(SL_{n-1}/[P,P])$ where $P$ is the parabolic subgroup $P_{(1,1,n-3)}$ of $SL_{n-1}(\C)$. Then we prove that the closure of the minimal nilpotent adjoint orbit $\Omin^{E_6}$ of the complex simple Lie algebra $\mathfrak{e}_6$ is isomorphic to the affinization of $T^*(SL_4/P^u)$ where $P^u$ is the unipotent radical of the parabolic subgroup $P_{(2,2)}$ of $SL_4(\C)$. In the end we will formulate a similar result for type $E_7$.
\end{abstract}

\section{Introduction}
Minimal nilpotent adjoint orbits of an complex simple Lie algebra are ubiquitous objects in the study of geometric representation theory. They have natural exact symplectic forms and their closure $\Ominbar$ are often related to $3d
\ \mathcal{N}=4$ Coulomb branches. In type $D_4$, there is an identification of $\Omin$ with the affinization of the total space of the cotangent bundle of $SL_3/U$: 
\begin{equation}\label{D4}
    \Ominbar^{D_4}\cong T^*(SL_3/U)^{\mathrm{aff}}.
\end{equation}
For a long time, we believe that the Hamiltonian reduction construction of $\Ominbar$ (In physics terms, as Higgs branch) is essential for the proof of such an identification. But now we have realized that there is a generalization of the embeddings of Lie algebras
$$
\slc\hookrightarrow\sld\hookrightarrow\so_8,
$$
which allows us to to formulate a deformed version of (\ref{D4}) by applying the main results in the paper \cite{LSS88}.
In this paper we first review of well-known generalization of the above result in type $D_n$. Then we also prove an analogous result in type $E_6$ and $E_7$. More specifically, we prove that the closure of the minimal nilpotent adjoint orbit $\Omin^{E_6}$ of the complex simple Lie algebra $\mathfrak{e}_6$ is isomorphic to the affinization of $T^*(SL_4/P^u)$ where $P^u$ is the unipotent radical of the parabolic subgroup $P_{(2,2)}$ of $SL_4(\C)$. It should also be noted that Tom Gannon and Ben Webster have informed the author that they have a different proof and a generalization \cite{GW25} of the above result for type $E_6$ .

\textbf{Acknowledgments.} 
The author would like to deeply thank Shizhuo Yu for hosting him at the conference ``Poisson Geometry and Cluster Algebras" in Nankai University. The author would also like to thank Sam Evens, Baohua Fu, Victor Ginzburg, Yu Li, Jie Liu, Zihang Liu, Peng Shan, Wenbin Yan for their useful suggestions and comments.

\section{The minimal nilpotent orbit in type $D_n$}
 Let $n\geq4$. Fix the Eulidean inner product $(\ ,\ )$ on $\C^{2n}$ given by 
 $$
 (\mathbf{x},\mathbf{y})=\sum_{k=1}^n x_k\, y_{2n+1-k}.
 $$ 
 Then under the identification of $\so(2n,\C)=\Wedge^2\C^{2n}$, the minimal nilpotent adjoint orbit $\Omin^{D_n}$ is identified with the $SO_{2n}$-orbit of $e_1\wedge e_2$, so 
\[
    \Ominbar^{D_n}=\{\alpha\in\Wedge^2\C^{2n}\suchthat \alpha\textrm{ is decomposable and isotropic}\}
\]
as explained in Proposition 3.5 in \cite{Jia21}. 

 Let $P$ be the parabolic subgroup of $SL_{n-1}$ corresponding to the partition $(1,1,n-3).$
Let $[P, P]$ denote the commutator subgroup of $P$. Then by \cite{DKS13}
that the affinization of $T^{*}(SL_{n-1}/[P, P])^{\mathrm{aff}}$ is isomorphic to the Hamiltonian reduction
$$
    T^*(\Hom(\C,\C^2)\oplus\Hom(\C^2,\C^{n-1}))/\!/\!/SL_2\coloneqq \mu^{-1}_{SL_2}(0)/\!/SL_2.
$$
Now by applying the same argument as in the proof of Proposition 3.6 in \cite{Jia21} we obtain
\begin{thm}\label{Dmain}
    The affinization $T^*(SL_{n-1}/[P,P])^{\mathrm{aff}}$ is isomorphic to the closure $ \Ominbar^{D_n}$ of the minimal nilpotent adjoint orbit in the Lie algebra $\so_{2n}$.
\end{thm}
The rest of the this section is devoted to another explanation of this theorem. And this new perspective allows us to generalize the theorem to other types.
Fix the standard basis $e_1,\cdots,e_n\in\C^n$. 
Then the homogeneous space $SL_{n-1}/[P,P]$ is quasi-affine with affine closure $\overline{SL_{n-1}/[P,P]}=\Spec(\C[SL_{n-1}]^{[P,P]})$ and we have an embedding 
$$
    \iota: \overline{SL_{n-1}/[P,P]}\hookrightarrow \Wedge^2\C^{n-1}\oplus\C^{n-1}
$$
given by 
\begin{align*}
    \C[\Wedge^2\C^{n-1}\oplus\C^{n-1}]& \longrightarrow\C[SL_{n-1}]^{[P,P]}\\
    f=f_1\otimes f_2\quad& \longmapsto\big(g\mapsto f_1(ge_1\wedge ge_2)f_2(ge_1)\big)
\end{align*}
Using the Bourbaki labeling for Dynkin diagrams we fix embeddings of Lie algebras
$$
\mathfrak{sl}_{n-1}\xhookrightarrow{\varphi_1}\sln\xhookrightarrow{\varphi_2}\so_{2n}
$$
such that under the pullback maps
\begin{align*}
    & (\varphi_2|_{\mathfrak{h}^{A_{n-1}}})^*:
    \alpha^{D_n}_{i}\mapsto \alpha^{A_{n-1}}_{i}\in(\mathfrak{h}^{A_{n-1}})^*,\textrm{ for all }i\leq n-1.\\
      & (\varphi_1|_{\mathfrak{h}^{A_{n-2}}})^*:
    \alpha^{A_{n-1}}_{i}\mapsto \alpha^{A_{n-2}}_{i}\in(\mathfrak{h}^{A_{n-1}})^*,\textrm{ for all }i\leq n-2.
\end{align*}

Let $\mathfrak{r}$ be the maximal abelian ideal of $\so(2n)$ given by
$$
\mathfrak{r}=\bigoplus_{\alpha\geq\alpha_n}\Lieg_\alpha,
$$
Then $\mathfrak{r}$ is the nilpotent radical of a parabolic subalgebra $\Liep_n\subset D_n$ whose levi subalgebra has semisimple part precisely the image $\varphi_2(\mathfrak{sl}_{n-1})$.

\begin{Lemma}
        The abelian ideal $\mathfrak{r}$ is the second fundamental representation $\Wedge^{2}\C^n$ of $\varphi_2(\mathfrak{sl}_n)$ and decomposes as irreducible $\varphi_2(\varphi_1(\mathfrak{sl}_{n-1}))$-representations into
    \begin{equation}\label{rdecomp}
               \mathfrak{r}=\C^{n-1}\oplus\Wedge^2\C^{n-1}
    \end{equation}
    such that the $\Ominbar^{D_n}\cap\mathfrak{r}=\iota\big( \overline{SL_{n-1}/[P,P]}\big)$
\end{Lemma}
\begin{proof}
Identify $\so_{2n}$ with $\Wedge^2(\C^{2n})$, then the highest root vector $X_\theta$ is identified with $e_1\wedge e_2$. Use the inner product $(\,,\,)$ to identify $\C^{2n}$ with $\C^n\oplus (\C^n)^*$. Then we decompose 
$$
    \Wedge^2(\C^n\oplus (\C^n)^*)=\Wedge^2(\C^n)\oplus(\sln\oplus\C)\oplus \Wedge^2(\C^n)^*
$$
as representations of $\varphi_2(\sln)$ and the abelian ideal $\mathfrak{r}$ is identified with $\Wedge^2(\C^n)$. Now we can identify $\C^n=\C^{n-1}\oplus\C$ and further decompose 
$$\Wedge^2(\C^n)=\Wedge^2(\C^{n-1}\oplus\C)=\Wedge^2\C^{n-1}\oplus\C^{n-1}$$
as representations of $\varphi_2(\varphi_1(\mathfrak{sl}_{n-1}))$.
Now for simplicity, we use $\alpha_i$ to denote the simple root $\alpha_i^{D_n}$. 
Then the minimal nilpotent orbit $\Omin=Ad({SO_{2n}})X_\theta$, and $\Ominbar=\Omin\cup\{0\}$.
    Notice that under the identification (\ref{rdecomp}) 
    $$
        Ad\big(\exp(Y_{-(\alpha_2+\alpha_3+\cdots+\alpha_{n-1})})\big)X_\theta=e_1\wedge e_2+e_1\wedge e_n
    $$
    Since the $\varphi_2(\varphi_1(SL_{n-1}))$-orbit of $e_1\wedge e_2+e_1\wedge e_n$ in $\Wedge^2\C^{n-1}\oplus\C^{n-1}$ is contained in the smooth part of the orbital variety $\Ominbar^{D_n}\cap\mathfrak{r}=\overline{Ad (\varphi_2(SL_{n}))X_\theta}$, which is irreducible and 
    $$
        \dim \Ominbar^{D_n}\cap\mathfrak{r}=\dim(SL_{n-1}/[P,P])=2n-3,
    $$
so $\iota\big( \overline{SL_{n-1}/[P,P]}\big)=\Ominbar^{D_n}\cap\mathfrak{r}$.
\end{proof}
Now we are ready to present
\begin{proof}[Another explanation of Theorem \ref{Dmain}]
    Let $X\coloneqq Ad (\varphi_2(SL_n))X_\theta$.
    Let $J_0$ be the Joseph ideal \cite{Jos76} of $\mathcal{U}(\mathfrak{so}_{2n})$. By Theorem 5.2 and Corollary 5.3.A in \cite{LSS88} we have an isomorphism from $\mathcal{U}(\mathfrak{so}_{2n})/J_0$ to the ring of algebraic differential operators on $X$. Then we may take associated graded on both sides to get our result. (A careful chosen of Fourier transforms are required in order to get an embedding of $\mathfrak{so}_{2n}$ into the derivations in the ring of differential operators on $X$.) 
\end{proof}

\section{The minimal nilpotent orbit in type $E_6$}
All results in this section are motivated and inspired by statements in the paper \cite{LSS88}.
Let $P^u$ be the unipotent radical of the parabolic subgroup $P_{(2,2)}$ of $SL_4(\C)$. That is
$$
    P^u=\left\{\begin{pmatrix}
        1 && 0 && * && *\\
        0 && 1 && * && *\\
        0 && 0 && 1 && 0\\
        0 && 0 && 0 && 1
    \end{pmatrix}\right\}.
$$
Fix the standard basis $e_1,\cdots,e_4\in\C^4$. Let $e_1^*,\cdots,e_4^*$ denote the corresponding dual basis of $(\C^4)^*$
Then the homogeneous space $SL_4/P^u$ is quasi-affine with affine closure $\overline{SL_4/P^u}=\Spec(\C[SL_4]^{P^u})$ and we have an embedding 
$$
    \iota: \overline{SL_4/P^u}\hookrightarrow \C^4\oplus (\C^4)^*\oplus (\C^4)^* \oplus \C^4
$$
given by 
\begin{align*}
    \C[\C^4\oplus (\C^4)^*\oplus (\C^4)^* \oplus \C^4]& \longrightarrow\C[SL_4]^{P^u}\\
    f=f_1\otimes f_2\otimes f_3\otimes f_4\quad& \longmapsto\big(g\mapsto f_1(ge_1)f_2(e_4^*g^{-1})f_3(e_3^*g^{-1})f_4(ge_2)\big)
\end{align*}
Using the Bourbaki labeling for Dynkin diagrams
$$
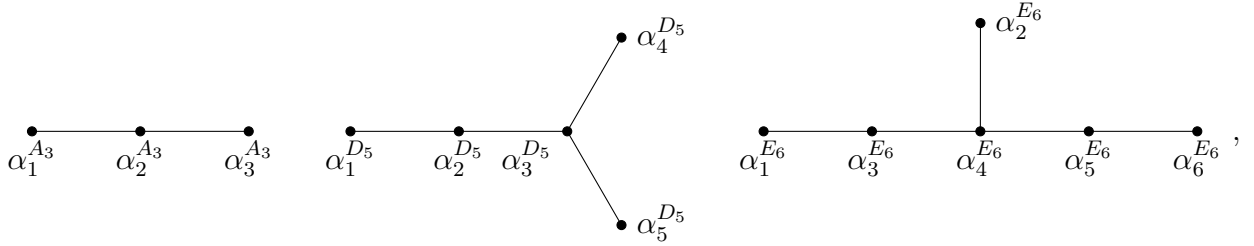
\dynkin[text style/.style={scale=1},label,label macro/.code={\alpha_{\drlap{#1}}^{A_3}},edge length=.75cm,scale=1.2,edge
length=1.2cm]A3\quad
\dynkin[text style/.style={scale=1},label,label macro/.code={\alpha_{\drlap{#1}}^{D_5}},edge length=.75cm,scale=1.2,edge
length=1.2cm]D5\quad
\dynkin[text style/.style={scale=1},label,label macro/.code={\alpha_{\drlap{#1}}^{E_6}},edge length=.75cm,scale=1.2,edge
length=1.2cm]E6,
$$
we fix embeddings of Lie algebras
$$
\mathfrak{sl}_4\xhookrightarrow{\varphi_1} \so_{10}\xhookrightarrow{\varphi_2} \mathfrak{e}_6
$$
such that under the pullback maps
\begin{align*}
    & (\varphi_2|_{\mathfrak{h}^{D_5}})^*:
    \alpha^{E_6}_{1}\mapsto 0\in(\mathfrak{h}^{D5})^*,\textrm{ and }
    \alpha^{E_6}_{i}\mapsto\alpha^{D_5}_{7-i}\textrm{ for }i\geq2.\\
     & (\varphi_1|_{\mathfrak{h}^{A_3}})^*:
        \alpha^{D_5}_{2}\mapsto \alpha^{A_3}_{1},
        \alpha^{D_5}_{4}\mapsto \alpha^{A_3}_{2},
        \alpha^{D_5}_{3}\mapsto \alpha^{A_3}_{3},        
        \text{ and other $\alpha^{D5}_{j}\mapsto 0\in(\mathfrak{h}^{A_3})^*$}.
\end{align*}
Let $\mathfrak{i}$ be the maximal abelian ideal of $E_6$ given by
$$
    \mathfrak{i}=\bigoplus_{\alpha\geq\alpha_1}\Lieg_\alpha,
$$
Then $\mathfrak{i}$ is the nilpotent radical of a parabolic subalgebra $\Liep_1\subset E_6$ whose levi subalgebra has semisimple part precisely the image $\varphi_2(\so_{10})$.
\begin{Lemma}
        The abelian ideal $\mathfrak{i}$ is the spin representation $\Wedge^{odd}\C^5$ of $\varphi_2(\so_{10})$ and decomposes as irreducible $\varphi_2(\varphi_1(\mathfrak{sl}_4))$-representations into
    \begin{equation}\label{idecomp}
               \mathfrak{i}=\C^4\oplus (\C^4)^*\oplus (\C^4)^* \oplus \C^4
    \end{equation}
    such that the $\Ominbar^{E_6}\cap\mathfrak{i}=\iota\big( \overline{SL_4/P^u}\big)$
\end{Lemma}
\begin{proof}
    The first statement follows from the fact that the highest root $\theta^{E_6}=\alpha^{E_6}_{1}+2\alpha^{E_6}_{2}+2\alpha^{E_6}_{3}+3\alpha^{E_6}_{4}+2\alpha^{E_6}_{5}+\alpha^{E_6}_{6}$ restricts to the highest weight $\alpha^{D_5}_{1}+2\alpha^{D_5}_{2}+3\alpha^{D_5}_{3}+2\alpha^{D_5}_{4}+2\alpha^{D_5}_{5}$ of of $\varphi_2(\so_{10})$ for the odd spin representation, and the roots
\begin{align*}
\alpha^{E_6}_{1}+2\alpha^{E_6}_{2}+2\alpha^{E_6}_{3}+3\alpha^{E_6}_{4}+2\alpha^{E_6}_{5}+\alpha^{E_6}_{6}\\
\alpha^{E_6}_{1}+\alpha^{E_6}_{2}+2\alpha^{E_6}_{3}+2\alpha^{E_6}_{4}+\alpha^{E_6}_{5}+\alpha^{E_6}_{6}\\
\alpha^{E_6}_{1}+\alpha^{E_6}_{2}+\alpha^{E_6}_{3}+2\alpha^{E_6}_{4}+\alpha^{E_6}_{5}\\
\alpha^{E_6}_{1}+\alpha^{E_6}_{3}+\alpha^{E_6}_{4}
\end{align*}
restricts to highest weights of $\varphi_2(\varphi_1(\mathfrak{sl}_4))$ of the irreps in the (RHS) of (\ref{idecomp}).
    Now for simplicity, we use $\alpha_i$ to denote the simple root $\alpha_i^{E_6}$. 
    Let $X_\theta\neq0$ be a highest root vector of $\mathfrak{e}_6$. Then $\Omin=Ad({E_6})X_\theta$, and $\Ominbar=\Omin\cup\{0\}$.
    Notice that under the identification (\ref{idecomp}) 
    $$
        Ad\big(\exp(Y_{-(\alpha_2+\alpha_3+\alpha_4+\alpha_5+\alpha_6)})\exp(Y_{-(\alpha_2+\alpha_4+\alpha_5)})\big)X_\theta=e_1\oplus e_4^*\oplus e_3^* \oplus e_2.
    $$
    Since the $\varphi_2(\varphi_1(SL_4))$-orbit of $e_1\oplus e_4^*\oplus e_3^* \oplus e_2$ in $\C^4\oplus (\C^4)^*\oplus (\C^4)^* \oplus \C^4$ is contained in the smooth part of orbital variety $\Ominbar^{E_6}\cap\mathfrak{i}=\overline{Ad (\varphi_2(Spin(10)))X_\theta}$, which is irreducible and 
    $$
        \dim \Ominbar^{E_6}\cap\mathfrak{i}=11=\dim(SL_4/P^u),
    $$
    so $\iota\big( \overline{SL_4/P^u}\big)=\Ominbar^{E_6}\cap\mathfrak{i}$.
\end{proof}
\begin{thm}
    The affinization $T^*(SL_4/P^u)^{\mathrm{aff}}$ is isomorphic to the closure $ \Ominbar^{E_6}$ of the minimal nilpotent adjoint orbit in the Lie algebra $\mathfrak{e}_6$.
\end{thm}
\begin{proof}
    Let $X\coloneqq Ad (\varphi_2(Spin(10)))X_\theta$.
    Let $J_0$ be the Joseph ideal \cite{Jos76} of $\mathcal{U}(\mathfrak{e}_6)$. By Theorem 5.2 and Corollary 5.3.A in \cite{LSS88} we have an isomorphism from $\mathcal{U}\mathfrak{e}_6/J_0$ to the ring of algebraic differential operators on $X$ 
    $$
        \overline\psi:\mathcal{U}(\mathfrak{e}_6)/J_0\xrightarrow{\cong}\mathcal{D}({X}).
    $$
    Let $\pi:\mathcal{U}(\mathfrak{e}_6)\rightarrow\mathcal{U}(\mathfrak{e}_6)/J_0$ be the quotient map and define the surjective ring homomorphism $$\psi\coloneqq \overline{\psi}\circ\pi:\mathcal{U}(\mathfrak{e}_6)\twoheadrightarrow\mathcal{D}({X}).$$
    Let $F^i(\mathcal{U}(\mathfrak{e}_6))$ be the pull back of the order filtration on $\mathcal{D}(X)$ under  $\psi$. 
    By Remark 3.4 (c) in \cite{LSS88}, for all $p\in X$, we have $\mathcal{O}_{X,p}$ and $\psi(\mathcal{U}(\varphi_2(\so_{10})))$ generate $\mathcal{D}_{X,p}$, and in our case this is due to $\varphi_2(\mathfrak{so}_{10})$ acts on $\mathcal{O}_X$ as derivations. Also by Remark 3.4 (a) in \cite{LSS88}, if we take $\mathfrak{r}^-$ as the $\mathrm{ad}\,\mathfrak{h}$-subrep of $\mathfrak{e}_6$ so that $$\mathfrak{e}_6=\mathfrak{p_1}\oplus\mathfrak{r}^-=\mathfrak{i}\oplus(\varphi_2(\so_{10})\oplus\C)\oplus \mathfrak{r}^-,$$ then the multiplication operators $\C[X]\subset \mathcal{D}(X)$ are generated by $\psi(\mathfrak{r}^-)$, and in our case since $\mathfrak{r}^-$ is abelian, indeed $\mathcal{U}(\mathfrak{r}^-)=\Sym(\mathfrak{r}^-)=\C[\mathfrak{i}]\twoheadrightarrow \C[X]$.
    So $\mathrm{gr}_{F^i}(\mathcal{U}(\mathfrak{e}_6))$ is finitely generated over $\mathrm{Sym}(\mathfrak{e}_6)$, that is
    $F^i$ is a good filtration. So $
        \Spec(\mathrm{gr}_{F^i}\,(\mathcal{U}(\mathfrak{e}_6)/J_0)) 
        $ is isomorphic to $\Ominbar^{E_6}$,
    the associated variety of the Joseph ideal.
    
     Since $\iota\big(\overline{SL_4/P^u}\big)=\overline{X}$ is normal, so we have the codimension of the complement of $SL_4/P^u$ is at least $2$ in $\overline{SL_4/P^u}$. So $\mathcal{D}({SL_4/P^u})=\mathcal{D}(X).$ Now with respect to the order filtration
    $$
\Spec(\mathrm{gr}\,\mathcal{D}(X))=\Spec(\mathrm{gr}\,\mathcal{D}(SL_4/P^u))
        =\Spec(\C[T^*(SL_4/P^u)])
        = T^*(SL_4/P^u)^{\mathrm{aff}},
    $$
    as desired.
    \end{proof}
    \begin{Corollary}
        The affinization $T^*(SL_4/P^u)^{\mathrm{aff}}$ has symplectic singularities in the sense of Beauville \cite{Bea00}.
    \end{Corollary}

\begin{Conjecture}
    The closure $\Ominbar^{E_6}$ is not isomorphic to any Hamiltonian reduction of the cotangent space $T^*V$ of a representation $V$ by a reductive complex algebraic group $G$.
\end{Conjecture}

By a similar argument one may also show that
\begin{Theorem}
    The affinization of the total space of the cotangent bundle of the $E_6$ orbit of a highest weight vector in the standard $27$-dimensional irreducible representation of $E_6$ is isomorphic to the closure of the minimal nilpotent orbit of $E_7$.
\end{Theorem}


\printbibliography
\bigskip
\footnotesize
 \noindent Boming Jia, \textit{Email}: \texttt{jiabm@tsinghua.edu.cn}\\
\textsc{Yau Mathematical Sciences Center,\\
 Jingzhai 301, Tsinghua University,\\
 Beijing, 100084, China.}\par\nopagebreak
 
\end{document}